\documentclass[reqno]{amsart}
\usepackage{amsmath,amsbsy,amsfonts}
\usepackage{color}

\usepackage[notref,notcite]{showkeys}
\usepackage[utf8]{inputenc}
\usepackage{graphicx}
\usepackage{amsmath,latexsym,amssymb}
\newtheorem{theorem}{Theorem}
\newtheorem{corollary}{Corollary}[section]

\newtheorem{lemma}{Lemma}[section]
\newtheorem{proposition}{Proposition}[section]
\newtheorem{remark}{Remark}[section]

\newcommand{\dd}{\mathrm{d}}

\numberwithin{equation}{section}

\begin{document}
	\title[Pseudo-relativistic Hartree equation]{Asymptotic behavior of ground states of generalized pseudo-relativistic Hartree equation}
	
	\author{P. Belchior}
	\address{P. Belchior - Pontifícia Universidade Católica, Departamento de Matemática, Av Dom José Gaspar, 500 - Coração Eucarístico, 30535-901  Belo Horizonte - MG, Brazil}
	\email{pedrobelchior18@gmail.com}
	\author{H. Bueno}
	\address{H. Bueno and G. A. Pereira- Departmento de Matem\'atica, Universidade Federal de  Minas Gerais, 31270-901 - Belo Horizonte - MG, Brazil}
	\email{hamilton.pb@gmail.com and  gilbertoapereira@yahoo.com.br }
	\author{O. H. Miyagaki}
	\address{O. H. Miyagaki - Departmento de Matem\'atica, Universidade Federal de Juiz de Fora, 36036-330 - Juiz de Fora - MG, Brazil}
	\email{ohmiyagaki@gmail.com}
	\author{G. A. Pereira}

	\thanks{O. H. Miyagaki is the corresponding author and has received research grants from CNPq/Brazil  304015/2014-8 and INCTMAT/CNPQ/Brazil; P. Belchior was partially supported by CAPES/Brazil: G. A. Pereira received research grants by PNPD/CAPES/Brazil}
	\date{}

\begin{abstract}
With appropriate hypotheses on the nonlinearity $f$, we prove the existence of a ground state solution $u$ for the problem
\[\sqrt{-\Delta+m^2}\, u+Vu=\left(W*F(u)\right)f(u)\ \ \text{in }\ \mathbb{R}^{N},\]
where $V$ is a bounded potential, not necessarily continuous, and $F$ the primitive of $f$. We also show that any of this problem is a classical solution. Furthermore, we prove that the ground state solution has exponential decay.
\end{abstract}
\maketitle

\section{Introduction}\label{intro}

In this paper we consider a generalized pseudo-relativistic Hartree equation
\begin{equation}\label{original}
\sqrt{-\Delta+m^2}\, u+Vu=\left(W*F(u)\right)f(u)\ \ \text{in }\ \mathbb{R}^{N},\end{equation}
where $N\geq 2$, $F(t)=\int_0^t f(s)\dd s$, assuming that the nonlinearity $f$ is a $C^1$ function, non-negative in $[0,\infty)$, that satisfies
\begin{enumerate}
\item [($f_1$)] $\displaystyle\lim_{t\to 0}\frac{|f(t)|}{t}=0$;
\item [($f_2$)] $\displaystyle\lim_{t\to \infty}\frac{f(t)}{t^{\theta-1}}=0$ for some $2<\theta<2^{\#}=\frac{2N}{N-1}$;
\item [($f_3$)] $\displaystyle\frac{f(t)}{t}$ is increasing for all $t>0$.
\end{enumerate}
We also postulate
\begin{enumerate}
	\item [($V1$)] $V$ is continuous and satisfies $V(y)+V_0\geq 0$ for every $y\in\mathbb{R}^{N}$ and some constant $V_0\in (0,m)$;
	\item [($V2$)] $V_\infty=\displaystyle\lim_{|y|\to\infty}V(y)>0$;
	\item [($V3$)] $V(y)\leq V_\infty$ for all $y\in\mathbb{R}^N$, $V(y)\neq V_\infty$;
\item [($W_h$)] $0\leq W=W_1+W_2\in L^r(\mathbb{R}^{N})+L^\infty(\mathbb{R}^{N})$ is radial, with $r>\frac{N}{N(2-\theta)+\theta}$.
\end{enumerate}

Therefore, we aim to generalize the results obtained by Coti Zelati and Nolasco \cite{ZelatiNolasco} and Cingolani and Secchi \cite{Cingolani}. In the last paper, the authors have studied the equation
\[\sqrt{-\Delta+m^2}\, u+Vu=\left(W*u^\theta\right)|u|^{\theta-2}u,\]
supposing, additionally to our hypotheses, that the potential $V$ is continuous and has a horizontal asymptote for $N\geq 3$. If $k\in \mathbb{N}$, our work covers the case
\[W(x)=\frac{|x|^k}{1+|x|^k},\]
while the hypothesis $W(y)\to 0$ when $|y|\to \infty$ is explicitly assumed in \cite[Section 7]{Cingolani}. Furthermore, the homogeneity of the equation is a key ingredient in the proofs presented. So, applying different methods, we generalize \cite{Cingolani}. A careful reading of our paper will also show that it generalizes \cite{ZelatiNolasco}.

The equation
\begin{equation}\label{original1}\left\{
\begin{array}{c}
i \partial_t u =\sqrt{-\Delta+m^2}+ G(u)\ \ \text{in }\ \mathbb{R}^{N},
\\
u(x,0)=\phi(x),\ \ x\in \mathbb{R}^{N} \end{array} \right.\end{equation}
where $N\geq 2$, $G$ is a nonlinearity of Hartree type , $m>0$ denotes the mass of bosons in units, was used to describe the dynamics of pseudo-relativistic boson stars in astrophysics. See \cite{Cho,Elgart,CSS,Lieb2} for more details. For the study of semiclassical analysis of the non-relativistic Hartree equations we would like to quote the papers \cite{CCS,Frohlich,Moroz,Wei} and the recent work \cite{Cingolani2} as well. For the Hartree equation without external potential $V$, we cite \cite{Lieb2} for radial ground state solution,  \cite{Lenzmann} for uniqueness and nondegeneracy of ground state solutions, and \cite{ZelatiNolasco,ZelatiNolasco2} for the existence of positive and radially symmetric solutions. In \cite{Melgaard} is treated some Hartree problem imposing that the external potential $V$ is radial, while in \cite{Cingolani} this condition is dropped.

By considering an extension problem from $\mathbb{R}^{N}$ to $\mathbb{R}^{N+1}_+$, an alternative definition of $\sqrt{-\Delta+m^2}$ is well-known (see \cite{ZelatiNolasco} or \cite{Caffarelli}), so that equation \eqref{original} can be written as
\begin{equation}\label{P}
\left\{\begin{aligned}
-\Delta u +m^2u&=0, &&\mbox{in} \ \mathbb{R}^{N+1}_+,\\
-\displaystyle\frac{\partial u}{\partial x}(0,y)&=-V(y)u(0,y)+\left(W(y)*F(u(0,y))\right)f(u(0,y)) &&\mbox{in} \ \mathbb{R}^{N}.\end{aligned}\right.
\end{equation}
We summarize our results:
\begin{theorem}\label{t1}
Suppose that conditions \textup{($f_1$)-($f_3$)}, \textup{($V_1$)} and \textup{($W_h$)} are valid. Then, problem \eqref{P} has a non-negative ground-state solution $w\in H^1(\mathbb{R}^{N+1}_+)$.
\end{theorem}

\begin{theorem}\label{classical}
	Assuming that hypotheses already stated are satisfied by $f$, $V$ and $W$, any solution $v$ of problem \eqref{P} satisfies
	\[v\in C^{1,\alpha}(\mathbb{R}^{N+1}_+)\cap C^2(\mathbb{R}^{N+1}_+)\]
	and therefore is a classical solution of \eqref{P}.
\end{theorem}

We also prove that the ground station solution has exponential decay:
\begin{theorem}\label{t3} Let $w$ be the ground state solution obtained in Theorem \ref{t1}. Then $w(x,y)>0$ in $[0,\infty)\times\mathbb{R}^{N}$ and, for any $\alpha\in (V_0,m)$ there exists $C>0$ such that
	\[0<w(x,y)\leq Ce^{-(m-\alpha)\sqrt{x^2+|y|^2}}e^{\alpha x}\]
	for any $(x,y)\in [0,\infty)\times\mathbb{R}^{N}$.
	In particular,
	\[0<w(0,y)\leq Ce^{-\delta|y|},\quad\forall\ y\in\mathbb{R}^{N},\]
	where $0<\delta<m-V_0$.
\end{theorem}

The natural setting for problem \eqref{P} is the Sobolev space \[H^1(\mathbb{R}^{N+1}_+)=\left\{u\in L^2(\mathbb{R}^{N+1}_+)\,:\, \iint_{\mathbb{R}^{N+1}_+}|\nabla u|^2\dd x\dd y<\infty \right\}\]
endowed with the norm
\[\|u\|^2=\iint_{\mathbb{R}^{N+1}_+}\left(|\nabla u|^2+u^2\right)\dd x\dd y.\]

\noindent\textbf{Notation.} The norm in the space $\mathbb{R}^{N+1}_+$ will be denoted by $\|\cdot\|$. For all $q\in [1,\infty]$, we denote by $|\cdot|_q$ the norm in the space $L^q(\mathbb{R}^{N})$ and by $\|\cdot\|_q$ the norm in the space $L^{q}(\mathbb{R}^{N+1}_+)$. \vspace*{.3cm}

It is well-known that traces of functions $H^1(\mathbb{R}^{N+1}_+)$ are in  $H^{1/2}(\mathbb{R}^{N})$ and that every function in $H^{1/2}(\mathbb{R}^{N})$ is the trace of a function in $H^1(\mathbb{R}^{N+1}_+)$, see \cite{Tartar}. Denoting $\gamma\colon H^1(\mathbb{R}^{N+1}_+)\to H^{1/2}(\mathbb{R}^{N})$ the linear function that associates the trace $\gamma(v)\in H^{1/2}(\mathbb{R}^{N})$ of the function $v\in H^1(\mathbb{R}^{N+1}_+)$, then $\ker\,\gamma=H^1_0(\mathbb{R}^{N+1}_+)$.

The immersions
\begin{align}\label{immersions}H^1(\mathbb{R}^{N+1}_+)&\hookrightarrow L^q(\mathbb{R}^{N+1}_+)\\
H^{1/2}(\mathbb{R}^{N})&\hookrightarrow L^q(\mathbb{R}^{N})\end{align}
are continuous for any $q\in [2,2^*]$ and $[2,2^{\#}]$ respectively, where
\begin{equation}\label{2*}2^{*}=\frac{2(N+1)}{N-1}\qquad\textrm{and}\qquad 2^{\#}=\frac{2N}{N-1}.\end{equation}

The space $H^{1/2}(\mathbb{R}^{N})$ is defined by means of Fourier transforms; therefore, we can not change $\mathbb{R}^{N}$ to a bounded open set $\Omega\subset\mathbb{R}^{N}$. However (see \cite{Demengel}),  $H^{1/2}(\mathbb{R}^{N})=W^{1/2,2}(\mathbb{R}^{N})$  and $W^{1/2,2}(\Omega)$ is well-defined for an open set $\Omega\subset\mathbb{R}^{N}$. We recall its definition. Let $u\colon \Omega\to \mathbb{R}$ a measurable function and $\Omega$ a bounded open set (that, in the sequel, we suppose to have Lipschitz boundary). Denoting
\[[u]^2_{\Omega}=\int_\Omega\int_\Omega\frac{|u(x)-u(y)|^2}{|x-y|^{N+1}}\dd x\dd y\]
and
\begin{align*}W^{1/2,2}(\Omega)&=\left\{u\in L^2(\mathbb{R}^{N})\,:\,[u]^2_{\Omega}<\infty\right\},
\end{align*}
then $W^{1/2,2}(\Omega)$ is a reflexive Banach space (see, e.g., \cite{Demengel} and \cite{Guide}) endowed with the norm
\[\|u\|_{W^{1/2,2}(\Omega)}=|u|_2+[u]_{\Omega}.\]\goodbreak

The proof of the next result can be found in \cite[Theorem 4.54]{Demengel}.
\begin{theorem}\label{immersionW}
	The immersion $W^{1/2,2}(\Omega)\hookrightarrow L^q(\Omega)$ is compact for any $q\in \left[1,2^{\#}\right)$.
\end{theorem}

As usual, the immersion $W^{1/2,2}(\Omega)\hookrightarrow L^{2^{\#}}(\Omega)$ is continuous: see \cite[Corollary 4.53]{Demengel}. We denote the norm in the space $L^q(\Omega)$ by $|\cdot|_{L^q(\Omega)}$.
\section{Preliminaries}
Let us suppose that  $u\in H^1(\mathbb{R}^{N+1})\cap C^\infty_0(\mathbb{R}^{N+1}_+)$ and $u(x,y)\geq 0$. Let us proceed heuristically: since
\[|u(0,y)|^t=\int_{\infty}^{0}\frac{\partial}{\partial x}|u(x,y)|^t\dd x=\int_{\infty}^{0}t|u(x,y)|^{t-2}u(x,y)\frac{\partial}{\partial x}u(x,y)\dd x,\]
it follows from Hölder's inequality
\begin{align}\label{Heu}\int_{\mathbb{R}^{N}}|\gamma(u)|^t=\int_{\mathbb{R}^{N}}|u(0,y)|^t\dd y&\leq \int_{\mathbb{R}^{N}}\int_0^\infty t|u(x,y)|^{t-1}|\nabla u(x,y)|\dd x\dd y\nonumber\\
&\leq t\left(\int_{\mathbb{R}^{N+1}_+}|u|^{2(t-1)}\right)^{1/2}\left(\int_{\mathbb{R}^{N+1}_+}|\nabla u|^2\right)^{1/2}\nonumber\\
&\leq t\|u\|_{2(t-1)}^{t-1}\|\nabla u\|_{2}.
\end{align}
So, in order to apply the immersion $H^1(\mathbb{R}^{N+1}_+)\hookrightarrow L^q(\mathbb{R}^{N+1}_+)$ we must have $2\leq 2(t-1)\leq \frac{2(N+1)}{N-1}$, that is, \begin{equation}\label{p}
2\leq t\leq\frac{2N}{N-1}=2^{\#}.
\end{equation}
By density of $H^1(\mathbb{R}^{N+1})\cap C^\infty_0(\mathbb{R}^{N+1}_+)$ in $H^1(\mathbb{R}^{N+1}_+)$, the estimate \eqref{Heu} is valid for all $u\in H^1(\mathbb{R}^{N+1}_+)$.

Taking into account \eqref{immersions}, Young's inequality applied to \eqref{Heu} yields
\begin{align}\label{casep}|\gamma(u)|_{t}&\leq \|u\|_{2(t-1)}^{(t-1)/t}\left(t\|\nabla u\|_{2}\right)^{1/t}\\
&\leq \frac{t-1}{t}\|u\|_{2(t-1)}+\|\nabla u\|_{2}\nonumber\\
&\leq C_t\|u\|,\nonumber
\end{align}
where $C_t$ is a constant. We summarize:
\begin{equation}\label{gammav}
|\gamma(u)|\in {L^t(\mathbb{R}^{N})},\ \ \forall\ t\in [2,2^{\#}].\end{equation}

The inequality \eqref{casep} will also be valuable in the special case $t=2$:
\begin{align}\label{p=2}
|\gamma(u)|^2_{2}&\leq \|u\|_{2}\left(2\|\nabla u\|_{2}\right)\nonumber\\
&\leq \lambda\iint_{\mathbb{R}^{N+1}_+}u^2+\frac{1}{\lambda}\iint_{\mathbb{R}^{N+1}_+}|\nabla u|^2
\end{align}
where $\lambda>0$ is a parameter, the last inequality being a consequence of Young's inequality.

\begin{remark}\label{obs1}
It follows from \textup{($f_3$)} that $f$ satisfies the Ambrosetti-Rabinowitz inequality $2F(t)\leq f(t)t$, for all $t>0$. Furthermore, it follows from \textup{($f_1$)} and \textup{($f_2$)} that, for any fixed $\xi>0$, there exists a constant $C_\xi$ such that
\begin{equation}\label{boundf}|f(t)|\leq\xi t+C_\xi t^{\theta-1},\quad\forall\ t\geq 0\end{equation}
and analogously
\begin{equation}\label{boundF}|F(t)|\leq\xi t^2+C_\xi t^{\theta}\leq C(t^2+t^\theta),\quad\forall\ t\geq 0.\end{equation}
Observe that $\gamma(u)\in L^\theta(\mathbb{R}^{N})$ and $\gamma(u)\in L^2(\mathbb{R}^{N})$ imply $F(\gamma(u))\in L^1(\mathbb{R}^{N})$.
\end{remark}

\begin{proposition}[Hausdorff-Young]\label{HYoung}
Assume that, for $1\leq p, q, s \leq \infty$, we have $f\in L^p (\mathbb{R}^{N})$, $g\in  L^q (\mathbb{R}^{N})$ and
\[\frac{1}{p}+\frac{1}{q}= 1 +\frac{1}{s}.\] Then
\[|f*g|_{s} \leq |f|_{p}|g|_{q}.\]
\end{proposition}

We now enhance the result given by \eqref{gammav}. Observe that $\frac{N}{N(2-\theta)+\theta}\geq 1$ and  $\frac{N}{N(2-\theta)+\theta}=1$ if, and only if $N=\theta=2$.

The results in the sequel will be useful when addressing the regularity of the solution of problem \eqref{P}.
\begin{lemma}\label{hipW} Concerning hypothesis $(W_h)$ we have:
	\begin{enumerate}
		\item [($i$)] if $r\in \left(\displaystyle\frac{N}{N(2-\theta)+\theta},\frac{2N}{N(2-\theta)+\theta}\right]$, there exists $\displaystyle p\in \left[1,\frac{2N}{(N-1)\theta}\right]$
		such that \[|\gamma(u)|^\theta\in L^p(\mathbb{R}^{N})\] and
		\[\frac{1}{p}+\frac{1}{r}=1+\frac{N(2-\theta)+\theta}{2N}.\]
		Furthermore, $F(\gamma(u))\in L^p(\mathbb{R}^{N})$ and
		\[|W_1*F(\gamma(u))|=:g\in {L^{2N/[N(2-\theta)+\theta]}(\mathbb{R}^{N})}.\]
		\item [($ii$)] if $r'$ denotes the conjugate exponent of $r$ and $r>\displaystyle\frac{2N}{N(2-\theta)+\theta}$, then $F(\gamma(u))\in L^{r'}(\mathbb{R}^{N})$ and $W_1*F(\gamma(u))\in L^\infty(\mathbb{R}^{N})$.
	\end{enumerate}
\end{lemma}

\noindent\begin{proof}($i$) We verify the values of $r$ that satisfy the equality \[\frac{1}{p}+\frac{1}{r}=1+\frac{N(2-\theta)+\theta}{2N}.\] Observe that $r\in \left(\frac{N}{N(2-\theta)+\theta},\frac{2N}{N(2-\theta)+\theta}\right]$ if, and only if, $p\in \left[1,\frac{2N}{(N-1)\theta}\right)$.
	
As consequence of \eqref{gammav} $|\gamma(u)|^\theta\in L^p(\mathbb{R}^{N})$ and thus $|\gamma(u)|^2\in L^p(\mathbb{R}^{N})$ and \eqref{boundF} yields $F(\gamma(u))\in L^p(\mathbb{R}^{N})$. So, $|W_1*F(\gamma(u))|=g\in {L^{2N/[N(2-\theta)+\theta]}(\mathbb{R}^{N})}$ follows from the Hausdorff-Young inequality.
	
($ii$) Since $W_1\in L^r(\mathbb{R}^{N})$ for $r=\frac{2N}{N(2-\theta)+\theta}$ and $r'=\frac{r}{r-1}=\frac{2N}{(N-1)\theta}$, applying ($i$) we conclude that $F(\gamma(u))\in L^{r'}(\mathbb{R}^{N})$ and $W_1*F(\gamma(u))\in L^\infty(\mathbb{R}^{N})$ is consequence of Proposition \ref{HYoung}.
$\hfill\Box$\end{proof}

\begin{corollary}\label{cor}We have
$|W*F(\gamma(u))|\leq C+g$ with $g\in L^{{2N/[N(2-\theta)+\theta]}}(\mathbb{R}^{N})$.
\end{corollary}

\noindent\begin{proof}An immediately consequence of Lemma \ref{hipW}, since $W_2\in L^\infty(\mathbb{R}^{N})$.
$\hfill\Box$\end{proof}\vspace*{.4cm}

Following arguments in \cite{ZelatiNolasco}, we have:
\begin{lemma}\label{c1} For all $\theta\in \left(2,\frac{2N}{N-1}\right)$, we have $|\gamma(u)|^{\theta-2}\leq 1+g_2$,	where $g_2\in L^N(\mathbb{R}^{N})$.
\end{lemma}

\noindent\begin{proof}We have
\[|\gamma(u)|^{\theta-2}=|\gamma(u)|^{\theta-2}\chi_{\{|\gamma(u)|\leq 1\}}+|\gamma(u)|^{\theta-2}\chi_{\{|\gamma(u)|>1\}}\leq 1+g_2,\]
with $g_2=|\gamma(u)|^{\theta-2}\chi_{\{|\gamma(u)|>1\}}$. If $(\theta-2)N\leq 2$, then
\[\int_{\mathbb{R}^{N}}|\gamma(u)|^{(\theta-2)N}\chi_{\{|\gamma(u)|>1\}}\leq \int_{\mathbb{R}^{N}}|\gamma(u)|^2\chi_{\{|\gamma(u)|>1\}}\leq\int_{\mathbb{R}^{N}}|\gamma(u)|^2<\infty.\]
	
When $2<(\theta-2)N$, then $(\theta-2)N\in \left(2,\frac{2N}{N-1}\right)$ and $|\gamma(u)|^{\theta-2}\in L^N(\mathbb{R}^{N})$ as outcome of \eqref{gammav}.
$\hfill\Box$\end{proof}

\begin{lemma}\label{c2} For all $\theta\in \left(2,\frac{2N}{N-1}\right)$ we have $h=g|\gamma(u)|^{\theta-2}\in L^N(\mathbb{R}^{N})$, where $g$ is the function of Lemma \ref{hipW}.
\end{lemma}

\noindent\begin{proof} Application of the Hölder inequality yields
\[\int_{\mathbb{R}^{N}}\left(g|\gamma(u)|^{\theta-2}\right)^N\leq \left(\int_{\mathbb{R}^{N}}g^{N\alpha}\right)^{\frac{1}{\alpha}}\left(\int_{\mathbb{R}^{N}}\left(|\gamma(u)|^{(\theta-2)N}\right)^{\alpha'}\right)^{\frac{1}{\alpha'}},\]
if we define $\alpha$ so that $\alpha N=2N/[N(2-\theta)+\theta]$. Thus, $\alpha'=2/[(N-1)(\theta-2)]$ and we have $\alpha'N(\theta-2)=2N/[N-1]$. Since both integrals of the right-hand side of the last inequality are integrable, we are done.
$\hfill\Box$\end{proof}\hspace*{.2cm}

We now handle the existence of the ``energy'' functional. We denote by $L^q_w(\mathbb{R}^{N})$ the weak $L^q(\mathbb{R}^{N})$ space and by $|\cdot|_{q_w}$ its usual norm (see \cite{Lieb}). The next result is a generalized version of the Hardy-Littlewood-Sobolev inequality:
\begin{proposition}[Lieb \cite{Lieb}]\label{pLieb} Assume that $p,q,r\in(1,\infty)$ and \[\frac{1}{p}+\frac{1}{q}+\frac{1}{r}=2.\]
	Then, for some constant $N_{p,q,t}>0$ and for any $f\in L^p(\mathbb{R}^{N})$, $g\in L^r(\mathbb{R}^{N})$ and
	$h\in L^q_w(\mathbb{R}^{N})$, we have the inequality
	\[\int_{\mathbb{R}^{N}}\int_{\mathbb{R}^{N}}f(t)h(t-s)g(s)\dd t\dd s\leq N_{p,q,t}|f|_{p}|g|_{r} |h|_{q_w}.\]
\end{proposition}

\begin{lemma}\label{estconv} For a positive constant $C$ holds
\[\left|\frac{1}{2}\int_{\mathbb{R}^{N}}\big(W*F(\gamma(u))\big)F(\gamma(u))\right|\leq  C\left(\|u\|^{2}+\|u\|^{\theta}\right)^2.\]
\end{lemma}

\noindent\begin{proof}Let us denote
\[\Psi(u)=\frac{1}{2}\int_{\mathbb{R}^{N}}\big[W*F(\gamma(u))\big]F(\gamma(u)).\]
Since $W=W_1+W_2$,
\begin{align}\label{I}\Psi(u)&=\frac{1}{2}\int_{\mathbb{R}^{N}}\big[W_1*F(\gamma(u))\big]F(\gamma(u))+\frac{1}{2}\int_{\mathbb{R}^{N}}\big[W_2*F(\gamma(u))\big]F(\gamma(u))\nonumber\\
&=:J_1(u)+J_2(u).\end{align}

Let us suppose that $|\gamma(u)|^\theta\in L^t(\mathbb{R}^{N})$ for some $t\geq 1$.
Then $|\gamma(u)|^2\in L^t(\mathbb{R}^{N})$ and $F(\gamma(u))\in L^t(\mathbb{R}^{N})$ (as consequence of \eqref{boundF}). Application of Proposition \ref{pLieb} yields
\[|J_1(u)|=\left|\frac{1}{2}\int_{\mathbb{R}^{N}}W_1*F(\gamma(u))\,F(\gamma(u))\right|\leq N\,|W_1|_{r}|F(\gamma(u))|_{t}|F(\gamma(u))|_t.\]
Since
$\frac{1}{r}+\frac{2}{t}=2$ implies $t=\frac{2r}{2r-1}$, we have
\begin{align}\label{J1} |J_1(u)| &\leq C|F(\gamma(u))|_{\frac{2r}{2r-1}}|F(\gamma(u))|_{\frac{2r}{2r-1}} \leq  C' ( \left\| u \right\|^2 +\left\| u \right|^\theta)^2 <\infty,
\end{align}
(Observe that, in order to apply the immersion $H^1(\mathbb{R}^{N+1}_+)\hookrightarrow L^q(\mathbb{R}^{N+1}_+)$, we must have $t\theta<2N/(N-1)$, that is,
$r>N/[N(2-\theta)+\theta]$.)

In the case $W_2\in L^\infty(\mathbb{R}^{N})$ we can take $t=1$, therefore
\begin{align}\label{J2}
|J_2(u)|&=\left|\frac{1}{2}\int_{\mathbb{R}^{N}}\big[W_2*F(\gamma(u))\big]F(\gamma(u))\right|\leq C\left(|\gamma(u)|^2_{2}+|\gamma(u)|^\theta_{\theta}\right)^2\nonumber\\
&\leq
C''\left(\|u\|^2+\|u\|^\theta\right)^2.
\end{align}

From \eqref{J2} and \eqref{J1} results the claim.
$\hfill\Box$\end{proof}\vspace*{.3cm}

\begin{lemma}\label{I1+I2}The functional
	\begin{align*}
	I(u) =&\frac{1}{2}\iint_{\mathbb{R}^{N+1}_+}\left(|\nabla u|^2+m^2u^2\right)+\frac{1}{2}\int_{\mathbb{R}^{N}}V(y)[\gamma(u(y))]^2\nonumber\\
	&\quad-\frac{1}{2}\int_{\mathbb{R}^{N}}\big(W*F(\gamma(u))\big)F(\gamma(u))\nonumber\\
	=&:I_1(u)+I_2(u)-\Psi(u)
	\end{align*}
	is well-defined.
\end{lemma}

\noindent\begin{proof}Of course
\begin{align*}
I_1(u)=\frac{1}{2}\iint_{\mathbb{R}^{N+1}_+}\left(|\nabla u|^2+m^2u^2\right)\leq \frac{k}{2}\|u\|^2<\infty,
\end{align*}
if we take $k=\max\{1,m^2\}$. Since hypothesis ($V_1$) implies $|V(y)|<C$, we have
\begin{align*}
|I_2(u)|=\left|\frac{1}{2}\int_{\mathbb{R}^{N}}V(y)[\gamma(u(y))]^2\right|\leq \frac{C}{2}\int_{\mathbb{R}^{N}}|\gamma(u)|^2=C'|\gamma(u)|^2_2\leq C''
\|u\|^2.
\end{align*}
Taking into account Lemma \ref{estconv}, the proof is complete.
$\hfill\Box$\end{proof}\vspace*{.2cm}

Since the derivative of the energy functional is given by
\begin{align}\label{derivative}
I'(u)\cdot \varphi =&\iint_{\mathbb{R}^{N+1}_+}\left[\nabla u\cdot\nabla \varphi+m^2u\varphi\right]+\int_{\mathbb{R}^{N}}V(y)\gamma(u)\gamma(\varphi)\nonumber\\
&\quad-\int_{\mathbb{R}^{N}}\left(W*F(\gamma(u))\right)f(\gamma(u))\gamma(\varphi),\ \forall\ \varphi\in H^1(\mathbb{R}^{N+1}_+),
\end{align}
we see that critical points of $I$ are weak solutions \eqref{P}.

Because we are looking for a positive solution, we suppose that $f(t)=0$ for $t<0$.

\begin{proposition}The quadratic form
	\[u\mapsto \frac{1}{2}\iint_{\mathbb{R}^{N+1}_+}\left(|\nabla u|^2+m^2u^2\right)+\frac{1}{2}\int_{\mathbb{R}^N}V(y)[\gamma(u(y))]^2\]
	defines an norm in the space $H^1(\mathbb{R}^{N+1}_+)$, which is equivalent to the norm $\|\cdot\|$.
\end{proposition}
\begin{proof}We keep up with the notation already introduced and note that
$I_2(u)\geq -(1/2)V_0\int_{\mathbb{R}^N}|\gamma(u)|^2$. Furthermore, as consequence of \eqref{p=2}, we have \begin{align}\label{g2a}\int_{\mathbb{R}^{N}}|\gamma(u)|^2\leq m\iint_{\mathbb{R}^{N+1}_+}|u|^2+\frac{1}{m}\iint_{\mathbb{R}^{N+1}_+}|\nabla u|^2.
\end{align}
Therefore,
\begin{align*}
I_1(u)+I_2(u)&\geq \frac{1}{2}\iint_{\mathbb{R}^{N+1}_+}\left(|\nabla u|^2+m^2u\right)-\frac{V_0m}{2}\iint_{\mathbb{R}^{N+1}_+}|u|^2-\frac{V_0}{2m}\iint_{\mathbb{R}^{N+1}_+}|\nabla u|^2\\
	&=\frac{1}{2}\left(1-\frac{V_0}{m}\right)\iint_{\mathbb{R}^{N+1}_+}|\nabla u|^2+\frac{1}{2}m(m-V_0)\iint_{\mathbb{R}^{N+1}_+}|u|^2.
	\end{align*}
	Defining $K=\min\left\{\frac{1}{2}\left(1-\frac{V_0}{m}\right),\frac{1}{2}m(m-V_0)\right\}>0$,
	we conclude that
	\[I_1(u)+I_2(u)\geq K\|u\|^2.\]
By applying \eqref{g2a} it easily follows that
\begin{align}\label{supbound}
I_1(u)+I_2(u)&\leq \frac{1}{2}\left(1+\frac{V_0}{m}\right)\iint_{\mathbb{R}^{N+1}_+}|\nabla u|^2+\frac{1}{2}\left(m^2+V_\infty m\right)\iint_{\mathbb{R}^{N+1}_+}|u|^2\nonumber\\
&\leq C\|u\|^2 \end{align}
for a constant $C>0$. We are done.
$\hfill\Box$\end{proof}

\section{Mountain pass geometry and Nehari manifold}\label{mpg}
\begin{lemma}\label{gpm}
$I$ satisfies the mountain pass theorem geometry. More precisely,
\begin{enumerate}
\item [$(i)$] There exist $\rho,\delta>0$ such that $I|_S\geq \delta>0$ for all $u\in S$, where
\[S=\left\{u\in H^1(\mathbb{R}^{N+1}_+)\,:\, \|u\|=\rho\right\}.\]
\item [$(ii)$] For each $u_0\in H^1(\mathbb{R}^{N+1}_+)$ such that $(u_0)_+\neq 0$, there exists  $\tau\in \mathbb{R}$, satisfying $\|\tau u_0\|>\rho$ and  $I(\tau u_0) <0$.
\end{enumerate}
\end{lemma}

\noindent \begin{proof}
%
%
Since we have already showed that
\begin{align}\label{I+}I_1(u)+I_2(u)\geq K\|u\|^2\end{align}
and so $I(u)\geq K\|u\|^2-\Psi(u)\geq K\|u\|^2-C\left(\|u\|^2+\|u\|^\theta\right)^2$, we obtain ($i$) by choosing $\rho>0$ small enough.

In order to prove ($ii$), fix $u_0\in H^1(\mathbb{R}^{N+1}_+)\setminus\{0\}$ such that $u_0\geq 0$. For all $t>0$ consider the function $g_{u_0}\colon(0,\infty)\to\mathbb{R}$ defined by
\[g_{u_0}(t)=\Psi\left(\frac{tu_0}{\|u_0\|}\right)\]
where, as before,
\[\Psi(u)=\frac{1}{2}\int_{\mathbb{R}^{N}}\big(W*F(\gamma(u))\big)F(\gamma(u)).\]

An easy calculation shows that
\begin{align*}
g'_{u_0}(t)&=\frac{2}{t}\int_{\mathbb{R}^{N}}
\left(W*F\left(\gamma\left(\frac{tu_0}{\|u_0\|}\right)\right)\right)\frac{f}{2}\left(\gamma\left(\frac{tu_0}{\|u_0\|}\right)\right)\gamma\left(\frac{tu_0}{\|u_0\|}\right)
\geq\frac{4}{t}g_{u_0}(t),\end{align*}
the last inequality being a consequence of the Ambrosetti-Rabinowitz inequality. Observe that $g'_{u_0}(t)>0$ for $t>0$.

Thus, we obtain
\begin{align*}
\ln g_{u_0}(t)\Big|_1^{\tau\|u_0\|}\geq 4\ln t\Big|_1^{\tau\|u_0\|}\quad\Rightarrow\quad \frac{g_{u_0}(\tau\|u_0\|)}{g_{u_0}(1)}\geq \left(\tau\|u_0\|\right)^{4},
\end{align*}
proving that
\begin{align}\label{H}
\Psi(\tau u_0)=g_{u_0}(\tau\|u_0\|)\geq D\left(\tau\|u_0\|\right)^{4}.\end{align}
for a constant $D>0$.

It follows from \eqref{supbound} that
\begin{align*}I(\tau u_0)&\leq C\tau^2\|u_0\|^2-D\tau^4\|u_0\|^4.
\end{align*}
Thus, it suffices to take $\tau$ large enough.
$\hfill\Box$\end{proof}\vspace*{.4cm}

The existence of a Palais-Smale sequence $(u_n)\subset H^1(\mathbb{R}^{N+1}_+)$ such that
\[I'(u_n)\to 0\qquad\textrm{and}\qquad I(u_n)\to c,\]
where
\[c=\inf_{\alpha\in \Gamma}\max_{t\in [0,1]}I(\alpha(t)),\]
and $\Gamma=\left\{\alpha\in C^1\left([0,1],H^1(\mathbb{R}^{N+1}_+)\right)\,:\,\alpha(0)=0,\,\alpha(1)<0\right\}$ results from the mountain pass theorem without the PS condition. \vspace*{.2cm}

We now consider the Nehari manifold
\begin{align*}\mathcal{N}&=\left\{u\in H^1(\mathbb{R}^{N+1}_+)\setminus\{0\}\,:\,I'(u)\cdot u=0\right\}.
\end{align*}
It is not difficult to see that $\mathcal{N}$ is a manifold in $H^1(\mathbb{R}^{N+1}_+)\setminus\{0\}$.

The next result, which follows immediately from our estimates, proves that $\mathcal{N}$ is a closed manifold in $H^1(\mathbb{R}^{N+1}_+)$:
\begin{lemma}\label{lN}
There exists $\beta>0$ such that $\|u\|\geq \beta$ for all $u\in \mathcal{N}$.
\end{lemma}

An alternative characterization of $c$ is obtained by a standard method: for $u_+\neq 0$, consider the function $\Phi(t)=I_1(tu)+I_2(tu)-\Psi(tu)$, preserving the notation of Lemma \ref{gpm}. The proof of Lemma \ref{gpm} assures that $\Psi(tu)>0$ for $t$ small enough, $\Psi(tu)<0$ for $t$ large enough and $g'_u(t)>0$ if $t>0$. Therefore, $\max_{t\geq 0}\Psi(t)$ is achieved at a unique $t_u=t(u)>0$ and $\Psi'(tu)>0$ for $t<t_u$ and $\Psi'(tu)<0$ for $t>t_u$. Furthermore, $\Psi'(t_uu)=0$ implies that $t_uu\in \mathcal{N}$.

The map $u\mapsto t_u$ ($u\neq 0$) is continuous and $c=c^*$, where
\[c^*=\inf_{u\in H^1(\mathbb{R}^{N+1}_+)\setminus\{0\}}\max_{t\geq 0} I(tu).\] For details, see \cite[Section 3]{Rabinowitz} or \cite{Felmer}.

Standard arguments prove the next affirmative:
\begin{lemma}\label{bounded}
Let $(u_n)\subset H^1(\mathbb{R}^{N+1}_+)$ be a sequence such that $I(u_n)\to c$ and $I'(u_n)\to 0$, where
\[c=\inf_{u\in H^1(\mathbb{R}^{N+1}_+)\setminus\{0\}}\max_{t\geq 0}I(tu).\]
Then $(u_n)$ is bounded and (for a subsequence) $u_n\rightharpoonup u$ in $H^1(\mathbb{R}^{N+1}_+)$.
\end{lemma}

\begin{lemma}\label{lK}
Let $U\subseteqq \mathbb{R}^{N}$ be any open set. For $1<p<\infty$, let $(f_n)$ be a bounded sequence in $L^p(U)$ such that $f_n(x)\to f(x)$ a.e. Then $f_n\rightharpoonup f$.
\end{lemma}
The proof of Lemma \ref{lK} can be found, e.g., in \cite[Lemme 4.8, Chapitre 1]{Kavian}.\vspace*{.2cm}

\section{The limit problem}
In this section we consider a variant of problem \eqref{P}, changing the potential $V(y)$ for $V_\infty$.
\begin{theorem} \label{teo ground state}
	Assuming $(f_1)$, $(f_2)$, $(f_3)$ and $(W_h)$, problem
	\begin{equation}\left\{\begin{array}{l}
	-\Delta u +m^2u=0 \ \ \text{in} \ \ \mathbb{R}^{N+1}_+ \\ \\
	\displaystyle-\frac{\partial u}{\partial x}= -V_\infty u+ \left[W*F(u)\right]f(u), \ (x,y)\in \left\{ 0\right\} \times \mathbb{R}^N \simeq\mathbb{R}^N,
	\end{array}\right.\tag{$P_\infty$}\label{Pinfty}
	\end{equation}
	has a non-negative ground state solution.
\end{theorem}

\begin{proof}Let $(u_n)$ be the minimizing sequence given by Lemma \ref{gpm}. Then, there exist $R,\delta>0$ and a sequence $(z_n)\subset\mathbb{R}^{N}$ such that \begin{equation}\label{Lions}\liminf_{n\to\infty}\int_{B_R(z_n)}|\gamma(u_n)|^2\geq \delta.\end{equation}
If false, a result of Lions (see \cite{CC}) guarantees that $\gamma(u_n)\to 0$ in $L^q(\mathbb{R}^{N})$ for $2<q<2^*$, thus implying that
\[\int_{\mathbb{R}^N} (W*F(\gamma(u_n)))f(\gamma(u_n)) \gamma(u_n)  \to 0,\] contradicting Lemma \ref{lN}.
	
We define
\[v_n(x)=u_n(x-z_n).\]
From \eqref{Lions} we derive that
\[\int_{B_R(0)}|\gamma(v_n)|^2\geq \frac{\delta}{2}.\]
We observe that the energy functional
\begin{align*}I_\infty(u) &= \frac{1}{2}  \iint_{\mathbb{R}^{N+1}_+} \left(|\nabla u|^2 + m^2u^2\right) +\frac{1}{2}\int_{\mathbb{R}^N} V_\infty|\gamma(u)|^2\\
&\qquad-\frac{1}{2}\int_{\mathbb{R}^N}\big[W*F(\gamma(u))\big]F(\gamma(u))
\end{align*}
and its derivative as well are translation invariant. Therefore, it also holds that
\[I_\infty'(v_n) \to 0\quad\textrm{ and }\quad I_\infty(v_n) \to c_\infty,\]
where
\[c_\infty=\inf_{u\in H^1(\mathbb{R}^{N+1}_+)\setminus\{0\}}\max_{t\geq 0} I_\infty(tu).\]
(Observe that all reasoning in Section \ref{mpg} is valid for $I_\infty$ and its minimizing sequence.)	
	
Since $(v_n)$ is bounded (see Lemma \ref{bounded}) it follows that $v_n\rightharpoonup v$. A standard argument shows that we can suppose $v_n(x)\to v(x)$ a.e. in  $(\mathbb{R}^{N+1}_+)$, $v_n\to v$ in $L^s_{loc}(\mathbb{R}^{N+1}_+)$ for all $s\in [2,2^*)$, $\gamma(v_n(x))\to \gamma(v(x))$ a.e. in $(\mathbb{R}^{N})$
and $\gamma(v_n)\to \gamma(v)$ in $L^q_{loc}(\mathbb{R}^{N})$, for all $q\in [p,p^{\#})$.

We will show that $v\in \mathcal{N}_\infty=\{u\in H^1(\mathbb{R}^{N+1_+})\setminus\{0\}\,:\,I'_\infty(u)\cdot u=0\}$.

For all $\varphi \in C^\infty_0(\mathbb{R}^{N+1}_+)$, let us consider $\psi_n=(v_n - v)\varphi\in H^1(\mathbb{R}^{N+1}_+)$. We have
\begin{align}\label{testfunction}
\langle I'_\infty(v_n) , \psi_n\rangle&= \iint_{\mathbb{R}^{N+1}_+} \nabla v_n \cdot \nabla\psi_n+ \iint_{\mathbb{R}^{N+1}_+}m^2v_n\psi_n+ \int_{\mathbb{R}^N} V_\infty\gamma(v_n) \gamma(\psi_n) \nonumber \\
{}&\qquad-\int_{\mathbb{R}^N} ( W * F(\gamma(v_n) ) f(\gamma(v_n)) \gamma(\psi_n)\nonumber\\
&=J_1+J_2+J_3-J_4.
\end{align}
We start considering
\begin{align*}
J_4=\int_{\mathbb{R}^N} ( W * F(\gamma(v_n) ) f(\gamma(v_n)) \gamma(\psi_n).
\end{align*}
Because $\displaystyle\lim_{n\to\infty}\langle I'_\infty(v_n) , (v_n - v)\varphi \rangle=0$, it follows from \cite[Lemma 3.5]{Ackermann} that $J_4\to 0$ when $n\to\infty$ and thus is easily verified that $J_2+J_3-J_4\to 0$ when $n\to\infty$. We now consider $J_1$:
\begin{align*}
J_1&=\iint_{\mathbb{R}^{N+1}_+} \nabla v_n \cdot \nabla ( (v_n-v)\varphi)\\
&=\iint_{\mathbb{R}^{N+1}_+} \nabla v_n \cdot  \varphi\nabla (v_n - v) +
\iint_{\mathbb{R}^{N+1}_+} \nabla v_n\cdot(v_n-v) \nabla \varphi\\
&=\iint_{\mathbb{R}^{N+1}_+}|\nabla(v_n-v)|^2\varphi+\varphi \nabla v\cdot \nabla(v_n-v)+\nabla v_n\cdot(v_n-v) \nabla \varphi.
\end{align*}
We infer that
\begin{align*}\lim_{n\to\infty}\iint_{\mathbb{R}^{N+1}_+} |\nabla(v_n-v)|^2 \varphi &=-\lim_{n\to\infty} \iint_{\mathbb{R}^{N+1}_+}
\varphi\nabla v \cdot\nabla(v_n-v) \\
&\qquad -  \lim_{n\to\infty}\iint_{\mathbb{R}^{N+1}_+} (v_n-v)\nabla v_n \cdot\nabla
\varphi.\end{align*}
Since
\[\lim_{n\to\infty} \iint_{\mathbb{R}^{N+1}_+} \varphi\nabla v\cdot \nabla( v_n - v) =0\ \text{ and }\
\lim_{n\to\infty}\iint_{\mathbb{R}^{N+1}_+} (v_n - v)\nabla v_n \cdot \nabla \varphi =0\]
(because $\nabla v_n$ is bounded), we deduce that
\[ \nabla v_n \rightarrow \nabla v \quad \ \mbox{a.e. in} \quad \mathbb{R}^{N+1}_+.\]
Thus
\[I'_\infty(v)v =0\]
and $v \in \mathcal{N}_\infty$.

We now turn our attention to the positivity of $v$. Seeing that
\[\iint_{\mathbb{R}^{N+1}_+}\left(\nabla v\cdot \nabla\varphi+m^2v\varphi\right)+\!\int_{\mathbb{R}^{N}}V_\infty\gamma(v)\gamma(\varphi)=\!\int_{\mathbb{R}^{N}}[W*F(\gamma(v))]f(\gamma(v))\gamma(\varphi)\]
and choosing $\varphi=v^-$, the left-hand side of the equality is positive (by the definition of $I_\infty$ and equation \eqref{I+} applied to $I_\infty$), since $J_1+J_2+J_3=I_1+I_2\geq K\|v\|^2$),  while $\Psi(v)=J_4\leq 0$. We are done.
$\hfill\Box$\end{proof}

\section{Proof of Theorem \ref{t1}}

In order to consider the general case of the potential $V(y)$, we state a well-known result due to M. Struwe:
\begin{lemma}[Splitting Lemma]\label{Struwe} Let $(v_n)\subset H^1(\mathbb{R}^{N+1}_+)$ be such that
	\[I(u_n)\to c,\qquad I'(u_n)\to 0\]
and $u_n\rightharpoonup u$ weakly on $X$. Then $I'(u_0)=0$ and we have \emph{either}
\begin{enumerate}
	\item [($i$)] $u_n\to u$ strongly on $X$;
	\item [($ii$)] there exist $k\in\mathbb{N}$, $(y^j_n)\in\mathbb{R}^N$ such that $|y^j_n|\to\infty$ for $j\in \{1,\ldots,k\}$ and nontrivial solutions $u^1,\ldots,u^k$ of problem \eqref{Pinfty} so that
	\[I(u_n)\to I(u_0)+\sum_{j=1}^k I_\infty (u_j)\]
	and
	\[\left\|u_n-u_0-\sum_{j=1}^ku^j(\cdot-y^j_n)\right\|\to 0.\]
\end{enumerate}
	
\end{lemma}

\begin{lemma}\label{PS}
	The functional $I$ satisfies $(PS)_c$ for any $0\leq c<c_\infty$.
\end{lemma}
\begin{proof}Let us suppose that $(u_n)$ satisfies
	\[I(u_n)\to c<c_\infty\qquad\text{and}\qquad I'(u_n)\to 0.\]
We can suppose that the sequence $(u_n)$ is bounded, according to Lemma \ref{bounded}. Therefore, for a subsequence, we have $u_n\hookrightarrow u_0$ in $H^1(\mathbb{R}^{N+1}_+)$. It follows from the Splitting Lemma (Lemma \ref{Struwe}) that $I'(u_0)=0$. Since
\begin{align*}
I'(u_0)\cdot u_0&=\iint_{\mathbb{R}^{N+1}_+}\left(|\nabla u_0|^2+m^2u^2_0\right)+\int_{\mathbb{R}^{N}}V(y)|\gamma(u_0)|^2\\
&\qquad-\int_{\mathbb{R}^{N}}[W*F(\gamma(u_0))]f(\gamma(u_0))\gamma(u_0)\\
\intertext{and}
I(u_0)&=\frac{1}{2}\iint_{\mathbb{R}^{N+1}_+}\left(|\nabla u_0|^2+m^2u^2_0\right)+\frac{1}{2}\int_{\mathbb{R}^{N}}V(y)|\gamma(u_0)|^2\\
&\qquad-\frac{1}{2}\int_{\mathbb{R}^{N}}[W*F(\gamma(u_0))]F(\gamma(u_0)),
\end{align*}
we conclude that
\begin{equation}\label{Iu0}
I(u_0)=\int_{\mathbb{R}^{N}}[W*F(\gamma(u_0))]\left(\frac{1}{2}f(\gamma(u_0))\gamma(u_0)-F(\gamma(u_0))\right)>0,\end{equation}
as consequence of the Ambrosetti-Rabinowitz condition.

If $u_n\not\to u$ in $H^1(\mathbb{R}^{N+1}_+)$, by applying again the Splitting Lemma we guarantee the existence of $k\in\mathbb{N}$ and nontrivial solutions $u^1,\ldots,u^k$ of problem \eqref{Pinfty} satisfying
\[\lim_{n\to\infty}I(u_n)=c=I(u_0)+\sum_{j=1}^kI_\infty(u^j)\geq kc_\infty\geq c_\infty\]
contradicting our hypothesis. We are done.
$\hfill\Box$\end{proof}\vspace*{.2cm}

We prove the next result by adapting the proof given in Furtado, Maia e Medeiros \cite{FMM}:
\begin{lemma}\label{ccinfty}Suppose that $V(y)$ satisfies $(V_3)$. Then
\[0<c<c_\infty,\]
where $c$ is characterized in Lemma \ref{bounded}.
\end{lemma}
\begin{proof}Let $\bar u\in \mathcal{N}_\infty$ be the weak solution of \eqref{Pinfty} given by Theorem \ref{teo ground state} and $t_{\bar u}>0$ be the unique number such that $t_{\bar u}\bar u\in \mathcal{N}$. We claim that $t_{\bar u}<1$. Indeed,
\[\int_{\mathbb{R}^{N}}[W*F(\gamma(t_{\bar u}\bar u))]f(\gamma(t_{\bar u}\bar u))\gamma(t_{\bar u}\bar u)\hspace*{7cm}\]

\vspace*{-.5cm}\begin{align*}
&=t^2_{\bar u}\iint_{\mathbb{R}^{N+1}_+}\left(|\nabla {\bar u}|^2+m^2{\bar u}^2\right)+\int_{\mathbb{R}^{N}}V(y)|\gamma(\bar{u})|^2\\
&< t^2_{\bar u}\iint_{\mathbb{R}^{N+1}_+}\left(|\nabla {\bar u}|^2+m^2{\bar u}^2\right)+\int_{\mathbb{R}^{N}}V_\infty|\gamma(\bar{u})|^2\\
&=t^2_{\bar u}\int_{\mathbb{R}^{N}}[W*F(\gamma(\bar u))]f(\gamma(\bar u))\gamma(\bar u)\\
&=t^2_{\bar u}\left(\int_{\mathbb{R}^{N}}[W*F(\gamma(\bar u))]f(\gamma(\bar u))\gamma(\bar u)+\int_{\mathbb{R}^{N}}[W*F(\gamma(t_{\bar u}\bar u))]f(\gamma(\bar u))\gamma(\bar u)\right.\\
&\qquad\quad\left.-\int_{\mathbb{R}^{N}}[W*F(\gamma(t_{\bar u}\bar u))]f(\gamma(\bar u))\gamma(\bar u)\right)
\end{align*}
thus yielding
\[\]
\begin{align*}
0&>\int_{\mathbb{R}^{N}}[W*F(\gamma(t_{\bar u}\bar u))]\left(\frac{f(\gamma(t_{\bar u}\bar u))}{\gamma(t_{\bar u}\bar u)}-\frac{f(\gamma(\bar u))}{\gamma(\bar u)}\right)\\
&\qquad +t^2_{\bar u}\int_{\mathbb{R}^{N}}\left[W*\left(F(\gamma(t_{\bar u}\bar u))-F(\gamma(\bar u))\right)\right]f(\gamma(u))\gamma(u).
\end{align*}

If $t_{\bar u}\geq 1$, since $f(s)/s$ is increasing, the first integral is non-negative and, since $F$ is increasing, the second integral as well. We conclude that $t_{\bar u}<1$.
	
Lemma \ref{bounded} and its previous comments show that
\[c\leq \max_{t\geq 0}I(t\bar u)=I(t_{\bar u}\bar u)=\int_{\mathbb{R}^{N}}[W*F(\gamma(t_{\bar u}\bar u))]\left(\frac{1}{2}f(\gamma(t_{\bar u}\bar u))\gamma(t_{\bar u}\bar u)-F(\gamma(t_{\bar u}\bar u))\right).\]
Since
\[g(t)=\int_{\mathbb{R}^{N}}[W*F(\gamma(t\bar u))]\left(\frac{1}{2}f(\gamma(t\bar u))\gamma(t\bar u)-F(\gamma(t\bar u))\right) \]
is a strictly increasing function, we conclude that
\[c=g(t_{\bar u})<g(1)=\int_{\mathbb{R}^{N}}[W*F(\gamma(\bar u))]\left(\frac{1}{2}f(\gamma(\bar u))\gamma(\bar u)-F(\gamma(\bar u))\right)=c_\infty,\]
proving our result. $\hfill\Box$\end{proof}

\noindent\textit{Proof of Theorem \ref{t1}}. Let $(u_n)$ be the minimizing sequence given by Lemma \ref{gpm}. It follows from Lemmas \ref{PS} and \ref{ccinfty} that $u_n\to u$ such that $I(u)=c$ and $I'(u)=0$.

We now turn our attention to the positivity of $u$. Seeing that
\[\iint_{\mathbb{R}^{N+1}_+}\left(\nabla u\cdot \nabla\varphi+m^2u\varphi\right)+\int_{\mathbb{R}^{N}}V(y)\gamma(u)\gamma\varphi=\int_{\mathbb{R}^{N}}[W*F(\gamma(u))]f(\gamma(u))\gamma(\varphi)\]
and choosing $\varphi=w^-$, the left-hand side of the equality is positive (by the definition of $I(u)$ and equation \eqref{I+}, since $I_1+I_2\geq K\|w\|^2$),  while $\Psi(u)\leq 0$. The proof is complete.
$\hfill\Box$

\section{Proof of Theorem \ref{classical}}

The proof of the next result adapts arguments in \cite{Cabre} and \cite{ZelatiNolasco}.
\begin{proposition}\label{p1} For all $\beta>0$ it holds
	\begin{align*}
	\hspace*{-.25cm}|\gamma(v_+)^{1+\beta}|^2_{2^{\#}}&\leq 2C^2_{2^{\#}}C_\beta\left[\left(|V|_\infty+CC_1(2+M)\right)|\gamma(v_+)^{1+\beta}|^2_{2}\right.\nonumber\\
	&\qquad\qquad\quad\left.+C_1|g|_{2N/[N(2-\theta)+\theta]}|\gamma(v_+)^{1+\beta}|^2_{2^{\#}(2/\theta)}\right],
	\end{align*}
	where $C_\beta=\max\{m^{-2},\left(1+\frac{\beta}{2}\right)\}$, $C,C_1,\tilde{C}$ and $M=M(\beta)$ are positive constants and $g=|W_1*F(\gamma(v))|$ is the function given by Lemma \ref{hipW}.
\end{proposition}

\noindent\begin{proof}
Choosing $\varphi=\varphi_{\beta,T}=vv^{2\beta}_T$ in \eqref{derivative}, where $v_T=\min\{v_+,T\}$ and $\beta>0$, we have $0\leq \varphi_{\beta,T}\in H^1(\mathbb{R}^{N+1}_+)$ and
\begin{multline}\label{varphibetaT}
\iint_{\mathbb{R}^{N+1}_+}\nabla v\cdot\nabla \varphi_{\beta,T}+m^2v\varphi_{\beta,T}\\
=-\int_{\mathbb{R}^{N}}V(y)\gamma(v)\gamma(\varphi_{\beta,T})+\int_{\mathbb{R}^{N}}\left(W*F(\gamma(v))\right)f(\gamma(v))\gamma(\varphi_{\beta,T}),
\end{multline}
Since $\nabla\varphi_{\beta,T}=v^{2\beta}_T\nabla v+2\beta vv^{2\beta-1}_T\nabla v_T$,
the left-hand side of \eqref{varphibetaT} is given by
\begin{multline}\label{varphibetaTl}
\iint_{\mathbb{R}^{N+1}_+}\nabla v\cdot \left(v^{2\beta}_T\nabla v+2\beta vv^{2\beta-1}_T\nabla v_T\right)+m^2v\left(vv^{2\beta}_T\right)\\
=\iint_{\mathbb{R}^{N+1}_+}v^{2\beta}_T\left[|\nabla v|^2+m^2v^2\right]+2\beta\iint_{D_T}v^{2\beta}_T|\nabla v|^2,
\end{multline}
where $D_T=\{(x,y)\in (0,\infty)\times \mathbb{R}^{N}\,:\, v_T(x,y)\leq T\}$.

Now we express \eqref{varphibetaTl} in terms of $\|vv^\beta_T\|^2$. For this, we note that $\nabla(vv^\beta_T)=v^\beta_T\nabla v+\beta vv^{\beta-1}_T\nabla v_T$. Therefore,
\[\iint_{\mathbb{R}^{N+1}_+}|\nabla (vv^\beta_T)|^2=\iint_{\mathbb{R}^{N+1}_+}v^{2\beta}_T|\nabla v|^2+(2\beta+\beta^2)\iint_{D_T}v^{2\beta}_T|\nabla v|^2,\]
thus yielding
\begin{align}\label{norm}
\|vv^\beta_T\|^2
&=\left(\iint_{\mathbb{R}^{N+1}_+}v^{2\beta}_T|\nabla v|^2+(2\beta+\beta^2)\iint_{D_T}v^{2\beta}_T|\nabla v|^2\right)+\iint_{\mathbb{R}^{N+1}_+}(vv^\beta_T)^2\nonumber\\
&=\iint_{\mathbb{R}^{N+1}_+}v^{2\beta}_T\left(|\nabla v|^2+|v|^2\right)+2\beta\left(1+\frac{\beta}{2}\right)\iint_{D_T}v^{2\beta}_T|\nabla v|^2\nonumber\\
&\leq C_\beta\left[\iint_{\mathbb{R}^{N+1}_+}v^{2\beta}_T\left(|\nabla v|^2+m^2|v|^2\right)+2\beta\iint_{D_T}v^{2\beta}_T|\nabla v|^2\right],
\end{align}
where $C_\beta=\max\left\{m^{-2},\left(1+\frac{\beta}{2}\right)\right\}$. Gathering \eqref{varphibetaT}, \eqref{varphibetaTl} and \eqref{norm}, we obtain
\begin{align}\label{norm=r}
\|vv^\beta_T\|^2\leq& C_\beta\left[-\int_{\mathbb{R}^{N}}V\gamma(v)^2\gamma(v_T)^{2\beta}\right.\nonumber\\
&\qquad+\left.\int_{\mathbb{R}^{N}}\left(W*F(\gamma(v))\right)f(\gamma(v))\gamma(v)\gamma(v_T)^{2\beta}\right].
\end{align}

We now start to consider the right-hand side of \eqref{norm=r}. Since $|f(t)|\leq C_1(|t|+|t|^{\theta-1})$, Corollary \ref{cor} shows that it can be written as
\begin{align}\label{rhs}
\leq& C_\beta\left[|V|_\infty\int_{\mathbb{R}^{N}}\gamma(vv_T^\beta)^{2}+\int_{\mathbb{R}^{N}}(C+g)|f(\gamma(v))|\,|\gamma(v)|\gamma(v_T)^{2\beta}\right]\nonumber\\
\leq&C_\beta\left[|V|_\infty\int_{\mathbb{R}^{N}}\gamma(vv_T^\beta)^{2}+C\int_{\mathbb{R}^{N}}C_1\left(|\gamma(v)|+|\gamma(v)|^{\theta-1}\right)|\gamma(v)|\gamma(v_T)^{2\beta}\right.\nonumber\\
&\qquad+\left.C_1\int_{\mathbb{R}^{N}}g\left(|\gamma(v)|+|\gamma(v)|^{\theta-1}\right)\,|\gamma(v)|\gamma(v_T)^{2\beta}\right]\nonumber\\
\leq& C_\beta\left[\left(|V|_\infty+CC_1\right)\int_{\mathbb{R}^{N}}\gamma(vv_T^\beta)^{2}+CC_1\int_{\mathbb{R}^{N}}|\gamma(v)|^{\theta-2}\gamma(v)^2\gamma(v_T)^{2\beta}\right.\nonumber\\
&\qquad+\left.C_1\int_{\mathbb{R}^{N}}g\gamma(vv_T^\beta)^{2}+C_1\int_{\mathbb{R}^{N}}g|\gamma(v)|^{\theta-2}\gamma(vv_T^\beta)^{2}\right].
\end{align}

Applying Lemmas \ref{c1} and \ref{c2}, inequality \eqref{rhs} becomes
\begin{align*}
\leq&C_\beta\left[\left(|V|_\infty+CC_1\right)\int_{\mathbb{R}^{N}}\gamma(vv_T^\beta)^{2}+CC_1\int_{\mathbb{R}^{N}}\left(1+g_2\right)\gamma(vv_T^\beta)^{2}\right.\nonumber\\
&\qquad+\left.C_1\int_{\mathbb{R}^{N}}g\gamma(vv_T^\beta)^{2}+C_1\int_{\mathbb{R}^{N}}h\gamma(vv_T^\beta)^{2}\right]\nonumber\\
\leq&C_\beta\left[\left(|V|_\infty+2CC_1\right)\!\!\int_{\mathbb{R}^{N}}\gamma(vv_T^\beta)^{2}+CC_1\int_{\mathbb{R}^{N}}g\gamma(vv_T^\beta)^{2}+CC_1\int_{\mathbb{R}^{N}}G\gamma(vv_T^\beta)^{2}\right],
\end{align*}
where $G=g_2+h\in L^N(\mathbb{R}^{N})$, admitting that $CC_1\geq C_1$.

Because $|\gamma(u)|_{2^\#}\leq C_{2^{\#}}\|u\|$ for all $u\in H^1(\mathbb{R}^{N+1}_+)$, the last inequality is equivalent to
\begin{align}\label{rhs3}
|\gamma(vv^{\beta}_T)|^2_{2^{\#}}&\leq C^2_{2^{\#}}C_\beta\left[\left(|V|_\infty+2CC_1\right)\!\!\int_{\mathbb{R}^{N}}\gamma(vv_T^\beta)^{2}+CC_1\int_{\mathbb{R}^{N}}g\gamma(vv_T^\beta)^{2}\right.\nonumber\\
&\qquad\qquad\left. +CC_1\int_{\mathbb{R}^{N}}G\gamma(vv_T^\beta)^{2}\right].
\end{align}

Let us consider the last integral in the right-hand side of \eqref{rhs3}. For all $M>0$, define $A_1=\{G\leq M\}$ and $A_2=\{G>M\}$. Then, whereas $G\in L^N(\mathbb{R}^{N})$,
\begin{align*}
\int_{\mathbb{R}^{N}}G\gamma(vv_T^\beta)^{2}\leq& M\int_{A_1}\gamma(vv_T^\beta)^{2}+\left(\int_{A_2}G^N\right)^{\frac{1}{N}}\left(\int_{A_2}\gamma(vv_T^\beta)^{2\frac{N}{N-1}}\right)^{\frac{N-1}{N}}\nonumber\\
\leq&M\int_{\mathbb{R}^{N}}\gamma(vv_T^\beta)^{2}+\epsilon(M)\left(\int_{\mathbb{R}^{N}}\gamma(vv_T^\beta)^{2^{\#}}\right)^{\frac{N-1}{N}},
\end{align*}
and $\epsilon(M)=\left(\int_{A_2}G^N\right)^{1/N}\to 0$ when $M\to\infty$.

If $M$ is taken so that $\epsilon(M)C^2_{2^{\#}}C_\beta CC_1<1/2$, we have
\begin{align}\label{rhs4}
|\gamma(vv^{\beta}_T)|^2_{2^{\#}}&\leq 2C^2_{2^{\#}}C_\beta
\left[\left(|V|_\infty+CC_1(2+M)\right)\int_{\mathbb{R}^{N}}\gamma(vv_T^\beta)^{2}\right.\nonumber\\
&\qquad\qquad\quad\left.+CC_1\int_{\mathbb{R}^{N}}g\gamma(vv_T^\beta)^{2}\right].
\end{align}

The Hölder inequality guarantees that
\begin{align*}
\int_{\mathbb{R}^{N}}g\gamma(vv^{\beta}_T)^{2}\leq |g|_{2N/[N(2-\theta)+\theta]}\left(\int_{\mathbb{R}^{N}}\gamma(vv^{\beta}_T)^{2\alpha'}\right)^{1/\alpha'},
\end{align*}
where
\[\alpha'=\frac{\frac{2N}{N(2-\theta)+\theta}}{\frac{2N}{N(2-\theta)+\theta}-1}=\frac{2N}{(N-1)\theta}=\frac{2^{\#}}{\theta}.\]
Thus,
\begin{align*}
\int_{\mathbb{R}^{N}}g\gamma(vv^{\beta}_T)^{2}\leq |g|_{2N/[N(2-\theta)+\theta]}\,|\gamma(vv^{\beta}_T)|^2_{2^{\#}(2/\theta)}
\end{align*}
and substitution on the right-hand side of \eqref{rhs4} yields
\begin{align}\label{rhs5}
\hspace*{-.25cm}|\gamma(vv^{\beta}_T)|^2_{2^{\#}}&\leq 2C^2_{2^{\#}}C_\beta\left[\left(|V|_\infty+CC_1(2+M)\right)|\gamma(vv^{\beta}_T)|^2_{2}\right.\nonumber\\
&\qquad\qquad\quad\left.+C_1|g|_{2N/[N(2-\theta)+\theta]}|\gamma(vv^{\beta}_T)|^2_{2^{\#}(2/\theta)}\right].
\end{align}

Since $vv^\beta_T\to v^{1+\beta}_+$, it follows from \eqref{rhs5} that
\begin{align*}
\hspace*{-.25cm}|\gamma(v_+)^{1+\beta}|^2_{2^{\#}}&\leq 2C^2_{2^{\#}}C_\beta\left[\left(|V|_\infty+CC_1(2+M)\right)|\gamma(v_+)^{1+\beta}|^2_{2}\right.\\
&\qquad\qquad\quad\left.+C_1|g|_{2N/[N(2-\theta)+\theta]}|\gamma(v_+)^{1+\beta}|^2_{2^{\#}(2/\theta)}\right],
\end{align*}
and we are done. (Observe, however, that $M$ depends on $\beta$.) $\hfill\Box$\end{proof}\vspace*{.2cm}

\begin{proposition}\label{p2}
	For all $p\in [2,\infty)$ we have $\gamma(v)\in L^p(\mathbb{R}^{N})$.
\end{proposition}

\noindent\begin{proof} Since $\frac{2N}{N-1}\frac{2}{\theta}\leq 2$ never occurs, we have $2<\frac{2^{\#}2}{\theta}=\frac{2N}{N-1}\frac{2}{\theta}<2^{\#}$.

According to the Proposition \ref{p1}, we have
\begin{align}\label{bs1}
|\gamma(v_+)^{1+\beta}|^2_{2^{\#}}&\leq \left[D_1|\gamma(v_+)^{1+\beta}|^2_{2}+E_1|\gamma(v_+)^{1+\beta}|^2_{2^{\#}(2/\theta)}\right],
\end{align}
where $D_1$ and $E_1$ are positive constants.

Choosing $\beta_1+1:=(\theta/2)>1$, it follows from \eqref{gammav} that
\[|\gamma(v_+)^{1+\beta}|^{2}_{2^{\#}(2/\theta)}=|\gamma(v_+)|^{\theta}_{\frac{2N}{N-1}}<\infty,\]
from what follows that the right-hand side of \eqref{bs1} is finite. We conclude that $|\gamma(v_+)|\in {L^{\frac{2N}{N-1}\frac{\theta}{2}}}(\mathbb{R}^{N})<\infty$. Now, we choose $\beta_2$ so that $\beta_2+1=(\theta/2)^2$ and conclude that
\[|\gamma(v_+)|\in L^{\frac{2N}{N-1}\frac{\theta^2}{2^2}}(\mathbb{R}^{N}).\]

After $k$ iterations we obtain that
\[|\gamma(v_+)|\in L^{\frac{2N}{N-1}\frac{\theta^k}{2^k}}(\mathbb{R}^{N}),\]
from what follows that $\gamma(v_+)\in L^p(\mathbb{R}^{N})$ for all $p\in [2,\infty)$. Since the same arguments are valid for $v_-$, we have $\gamma(v)\in L^p(\mathbb{R}^{N})$ for all $p\in [2,\infty)$.
$\hfill\Box$\end{proof}\vspace*{.2cm}

By simply adapting the proof given in \cite{ZelatiNolasco}, we present, for the convenience of the reader, the proof of our next result:\vspace*{.2cm}

\begin{proposition}\label{t2}
	Let $v\in H^1(\mathbb{R}^{N+1}_+)$ be a weak solution of \eqref{P}. Then $\gamma(v)\in L^p(\mathbb{R}^{N})$ for all $p\in [2,\infty]$ and $v\in L^\infty(\mathbb{R}^{N+1}_+)$.
\end{proposition}
\noindent\begin{proof} We recall equation \eqref{norm=r}:
\begin{align*}
\|vv^\beta_T\|^2\leq& C_\beta\left[-\int_{\mathbb{R}^{N}}V\gamma(vv_T^\beta)^2
+\int_{\mathbb{R}^{N}}\left(W*F(\gamma(v))\right)f(\gamma(v))\gamma(v)\gamma(v_T)^{2\beta}\right],
\end{align*}
where $C_\beta=\max\{m^{-2},(1+\beta^2)\}$.

It follows that  $W*F(\gamma(v))\in L^\infty(\mathbb{R}^{N})$, since $\gamma(v)\in L^p(\mathbb{R}^{N})$ for all $p\geq 2$, by Proposition \ref{p2}. We also know that $|f(t)|\leq C_1(|t|+|t|^{\theta-1})$ and $V$ is bounded. Therefore, if $C=\max\{|V|_\infty, C_1|W*F(\gamma)|_\infty\}$, we have
\begin{align*}
\|vv^\beta_T\|^2&\leq C_\beta C\left[\int_{\mathbb{R}^{N}}\gamma(vv_T^\beta)^ 2+\int_{\mathbb{R}^{N}}\left(|\gamma(v)|+|\gamma(v)|^{\theta-1}\right)\gamma(v)\gamma(v_T)^{2\beta}\right]\\
&\leq C_\beta\left[2 C\int_{\mathbb{R}^{N}}\gamma(vv_T^\beta)^ 2+ C\int_{\mathbb{R}^{N}}|\gamma(v)|^{\theta-2}\gamma(vv_T^{\beta})^2\right].
\end{align*}
Since $|\gamma(v)|^{p-2}=|\gamma(v)|^{p-2}\chi_{\{|\gamma(v)\leq 1\}}+|\gamma(v)|^{p-2}\chi_{\{|\gamma(v)> 1\}}$,
the fact that \[|\gamma(v)|^{p-2}\chi_{\{|\gamma(v)> 1\}}=:g_3\in L^{2N}(\mathbb{R}^{N})\]
allows us to conclude that
\begin{align*}2 C\gamma(vv_T^\beta)^ 2+ C|\gamma(v)|^{\theta-2}\gamma(vv_T^{\beta})^2\leq (C_3+g_3)\gamma(vv_T^\beta)^2
\end{align*}
for a positive constant $C_3$ and a positive function $g_3\in L^{2N}(\mathbb{R}^{N})$ that depends neither on $T$ nor on $\beta$.

Therefore,
\begin{align*}
\|vv^\beta_T\|^2&\leq \int_{\mathbb{R}^{N}}(C_3+g_3)\gamma(vv_T^\beta)^2.
\end{align*}
and
\begin{align*}
\|v^{\beta+1}_+\|^2&\leq C_\beta\int_{\mathbb{R}^{N}}(C_3+g_3)\gamma(v^{\beta+1}_+)^2.
\end{align*}
Since
\begin{align*}
\int_{\mathbb{R}^{N}}g_3\gamma(v^{\beta+1}_+)^2&\leq |g_3|_{2N}\,|\gamma(v_+)^{1+\beta}|_2\, |\gamma(v_+)^{1+\beta}|_{2^{\#}}\\
&\leq |g_3|_{2N}\left(\lambda|\gamma(v_+)^{1+\beta}|^2_2+\frac{1}{\lambda}|\gamma(v_+)^{1+\beta}|^2_{2^{\#}}\right),
\end{align*}
we conclude that
\begin{align*}
|\gamma(v_+)^{1+\beta}|^2_{2^{\#}}&\leq C^2_{2^{\#}} \|v^{\beta+1}_+\|^2\\
&\leq C^2_{2^{\#}}C_\beta\left(C_3+\lambda\,|g_3|_{2N}\right)|\gamma(v_+)^{1+\beta}|^2_2+\frac{C^2_{2^{\#}}C_\beta\,|g_3|_{2N}}{\lambda}|\gamma(v_+)^{1+\beta}|^2_{2^{\#}}
\end{align*}
and, by taking $\lambda>0$ so that
\[\frac{C^2_{2^{\#}}C_\beta\,|g_3|_{2N}}{\lambda}<\frac{1}{2},\]
we obtain
\begin{align}\label{fest}
|\gamma(v_+)^{1+\beta}|^2_{2^{\#}}&\leq C_\beta\left(2C^2_{2^{\#}}C_3+2C^2_{2^{\#}}\lambda\,|g_3|_{2N}\right)|\gamma(v_+)^{1+\beta}|^2_2\nonumber\\
&\leq C_4C_{\beta}|\gamma(v_+)^{1+\beta}|^2_2.
\end{align}
Since
\[C_4C_\beta\leq C_4(m^{-2}+1+\beta)\leq M^2e^{2\sqrt{1+\beta}}\]
for a positive constant $M$, it follows from \eqref{fest} that
\begin{align*}
|\gamma(v_+)|_{2^{\#}(1+\beta)}&\leq M^{1/(1+\beta)}e^{1/\sqrt{1+\beta}}|\gamma(v_+)|_{2(1+\beta)}.
\end{align*}
We now apply an iteration argument, taking $2(1+\beta_{n+1})=2^{\#}(1+\beta_n)$ and starting with $\beta_0=0$. This produces
\[|\gamma(v_+)|_{2^{\#}(1+\beta_n)}\leq M^{1/(1+\beta_n)}e^{1/\sqrt{1+\beta_n}}|\gamma(v_+)|_{2(1+\beta_n)}.\]
Because $(1+\beta_n)=\left(\frac{2^{\#}}{2}\right)^n=\left(\frac{N}{N-1}\right)^n$,
we have
\[\sum_{i=0}^\infty \frac{1}{1+\beta_n}<\infty\qquad\textrm{and}\qquad \sum_{i=0}^\infty\frac{1}{\sqrt{1+\beta_n}}<\infty.\]

Thus,
\[|\gamma(v_+)|_\infty=\lim_{n\to\infty}|\gamma(v_+)|_{2^{\#}(1+\beta_n)}<\infty,\]
from what follows $|\gamma(v_+)|_p<\infty$ for all $p\in [2,\infty]$. The same argument applies to $\gamma(v_-)$, proving that $\gamma(v)\in L^p(\mathbb{R}^{N})$ for all $p\in [2,\infty]$.

By taking $\lambda=1$ and $|\gamma(v_+)^{1+\beta}|_{p}<C_5$ for all $p$, we obtain for any $\beta>0$,
\begin{align}\label{final0}
\|v^{\beta+1}_+\|^2
&\leq C_\beta\left(C_3+|g_3|_{2N}\right)C^{2}_5+C_\beta\,|g_3|_{2N}C^{2}_5.
\end{align}

But $\|v_+\|^{1+\beta}_{2^*(1+\beta)}=\|v_+^{1+\beta}\|_{2^*}\leq C_{2^*}\|v_+^{1+\beta}\|$ and for a positive constant $\tilde{c}$ results from \eqref{final0}  that
\[\|v_+\|^{2(1+\beta)}_{2^*(1+\beta)}\leq \tilde{c}C_\beta C^{2(1+\beta)}_5.\]
Thus,
\[\|v_+\|_{2^*(1+\beta)}\leq \tilde{c}^{1/2(1+\beta)}C_\beta^{1/2(1+\beta)}C_5 \]
and the right-hand side of the last inequality is uniformly bounded for all $\beta>0$. We are done.
$\hfill\Box$\end{proof}

We now state \cite[Proposition 3.9]{ZelatiNolasco}:
\begin{proposition}\label{regZN}Suppose that $v\in H^1(\mathbb{R}^{N+1}_+)\cap L^\infty(\mathbb{R}^{N+1}_+)$ is a weak solution of
\begin{equation}\label{C}\left\{\begin{aligned}
-\Delta v +m^2v&=0, &&\mbox{in} \ \mathbb{R}^{N+1}_+,\\
-\displaystyle\frac{\partial v}{\partial x}(0,y)&=h(y) &&\mbox{for all} \ y\in\mathbb{R}^{N},\end{aligned}\right.\end{equation}
where $h\in L^p(\mathbb{R}^{N})$ for all $p\in [2,\infty]$.

Then $v\in C^{\alpha}([0,\infty)\times\mathbb{R}^{N})\cap W^{1,q}((0,R)\times\mathbb{R}^{N})$ for all $q\in [2,\infty)$ and $R>0$.

In addition, if $h\in C^\alpha(\mathbb{R}^N)$, then $v\in C^{1,\alpha}([0,\infty)\times\mathbb{R}^N)\cap C^2(\mathbb{R}^{N+1}_+)$ is a classical solution of \textup{\eqref{C}}.
\end{proposition}

\noindent{\textit{Proof of Theorem \ref{classical}.} In the proof of Proposition \ref{regZN} (see \cite[Proposition 3.9]{ZelatiNolasco}), defining
	\[\rho(x,y)=\int_0^x v(t,y)\dd t,\]
	taking the odd extension of $h$ and $\rho$ to the whole $\mathbb{R}^{N+1}$ (which we still denote simply by $h$ and $\rho$), in \cite{ZelatiNolasco} is obtained that $\rho$ satisfies the equation
	\begin{equation}\label{rho}-\Delta\rho+m^2\rho=h\quad\text{in }\ \mathbb{R}^{N+1}\end{equation}
	and $\rho\in C^{1,\alpha}(\mathbb{R}^{N+1})$ for all $\alpha\in(0,1)$ by applying Sobolev's embedding. Therefore, $v(x,y)=\frac{\partial \rho}{\partial x}(x,y)\in C^\alpha(\mathbb{R}^{N})$.
	
In our case
\[h(y)=-V(y)v(0,y)+\left(W*F\left(v(0,y)\right)\right)f\left(v(0,y)\right).\]

We now rewrite equation \eqref{rho} as
\[-\Delta \rho+V(y)\frac{\partial \rho}{\partial x}(0,y)+m^2\rho=\left(W*F\left(\frac{\partial \rho}{\partial x}(0,y)\right)\right)f\left(\frac{\partial \rho}{\partial x}(0,y)\right).\]
Since $f\in C^1$ and $\frac{\partial \rho}{\partial x}(x,y)$ is bounded, the right-hand side of the last equality belongs to $C^\alpha(\mathbb{R}^{N+1})$. Thus, classical elliptic boundary regularity yields
 \[\rho\in C^{2}(\mathbb{R}^{N+1})\quad\Rightarrow\quad v\in C^{1,\alpha}(\mathbb{R}^N_+).\]
Hence, by applying classical interior elliptic regularity directly to $v$, we deduce that $v\in C^{1,\alpha}(\mathbb{R}^N_+)\cap C^{2}(\mathbb{R}^N_+)$ is a classical solution of problem \eqref{P}. $\hfill\Box$

\section{Proof of Theorem \ref{t3}}
We now adjust \cite[Theorem 3.14]{ZelatiNolasco} to our needs. The original statement guarantees that $v\in C^\infty([0,\infty)\times\mathbb{R}^{N})$, a result that depends on the function $h$ (of Proposition \ref{regZN}) considered in that paper. For the convenience of the reader, we present the proof of the next result:
\begin{theorem}\label{314}Suppose that $v\in H^1(\mathbb{R}^{N+1}_+)$ is a critical point of the energy functional $I$, then 
	\[|v(x,y)|e^{\lambda x}\to 0\]
	as $x+|y|\to \infty$, for any $\lambda<m$.
\end{theorem}

\noindent \begin{proof} Let us consider a solution $v$ of the problem
\[\left\{\begin{array}{ll}
-\Delta v+m^2v=0 &\text{in}\ \mathbb{R}^{N+1}_+\\
v(0,y)=v_0(y)\in L^2(\mathbb{R}^{N}), &y\in \mathbb{R}^{N}=\partial \mathbb{R}^{N+1}_+.
\end{array}\right.\]

By applying the Fourier transform with respect to variable $y\in\mathbb{R}^{N}$ we obtain
\[\mathcal{F}v(x,k)=e^{-\sqrt{|2\pi k|^2+m^2}\,x}\mathcal{F}v_0(y),\]
from what follows
\[\sup_{y\in\mathbb{R}^{N}}|v(x,y)|\leq C|v_0|_2e^{-mx}.\]
Since Proposition \ref{regZN} shows that $v\in W^{1,q}((0,R)\times\mathbb{R}^{N})$ for all $q\in [2,\infty)$ and $R>0$, we conclude that $|v(x,y)|\to 0$ when $|y|\to\infty$ for any $x$ and $|v(x,y)|e^{\lambda x}\to 0$ as $x+|y|\to \infty$ for any $\lambda<m$.
$\hfill\Box$\end{proof}\hspace*{.2cm}

We now adapt the proof of \cite[Theorem 5.1]{ZelatiNolasco}. In that paper is assumed that $W(y)\to 0$ as $|y|\to\infty$, a condition that is not necessary.\\

\noindent\textit{Proof of Theorem \ref{t3}}. We denote
\[K(y)=W*F\left(\frac{\partial w}{\partial x}(0,y)\right).\]
It follows easily that $K$ is bounded.

By Theorem \ref{t1} we have $w(x,y)\geq 0$. Applying Harnack's inequality we conclude that $w$ is strictly positive.

Following \cite{ZelatiNolasco}, for any $R>0$ we denote
\begin{align*}
B^+_R&=\{(x,y)\in \mathbb{R}^{N+1}_+\,:\, \sqrt{x^2+|y|^2}<R\}\\
\Omega^+_R&=\{(x,y)\in \mathbb{R}^{N+1}_+\,:\, \sqrt{x^2+|y|^2}>R\}\\
\Gamma^+_R&=\{(0,y)\in \mathbb{R}^{N+1}_+\,:\, |y|>R\}
\end{align*}
and define
\[f_R(x,y)=C_Re^{-\alpha x}e^{-(m-\alpha)\sqrt{x^2+|y|^2}},\]
where the positive constants $C_R$ and $\alpha\in (V_0,m)$ will be chosen later on. A simple computation shows that
\[\Delta f_R=\left(\alpha^2+(m-\alpha)^2+\frac{2\alpha(m-\alpha)x}{\sqrt{x^2+|y|^2}}-\frac{N(m-\alpha)}{\sqrt{x^2+|y|^2}}\right)f_R.\]
Thus, for $R$ large enough, we have
\[\left\{\begin{array}{ll}
-\Delta f_R+m^2f_R\geq 0 &\text{in }\ \Omega^+_R\\
-\frac{\partial f_R}{\partial x}=\frac{\partial f_R}{\partial \eta}=\alpha f_R &\text{on }\ \Gamma^+_R.\end{array}  \right.\]

We now define
\[\rho(x,y)=f_R(x,y)-w(x,y).\]
We clearly have $-\Delta\rho(x,y)-m^2\rho(x,y)\geq 0$ in $\Omega^+_R$. Choosing \[C_R=e^{mR}\max_{\partial B^+_R}v,\] we also have $\rho(x,y)\geq 0$ on $\partial B^+_R$ and $\rho(x,y)\to 0$ when $x+|y|\to\infty$.

We claim that $\rho(x,y)\geq 0$ in $\overline{\Omega}^+_R$. Supposing the contrary, let us assume that $\inf_{\overline{\Omega}^+_R} \rho(x,y)<0$. By the strong maximum principle, there exist $(0,y_0)\in \Gamma^+_R$ such that $\rho(0,y_0)=\inf_{\overline{\Omega}^+_R} \rho(x,y)<\rho(x,y)$ for all $(x,y)\in \Omega^+_R$. Defining
\[z(x,y)=\rho(x,y)e^{\lambda x}\]
for some $\lambda\in (V_0,m)$, a straightforward calculation shows that
\[-\Delta \rho+m^2\rho=e^{-\lambda x}\left(\-\Delta z+2\lambda \partial_xz+(m^2-\lambda^2)z\right).\]
Since $-\Delta \rho+m^2\rho\geq 0$, we conclude that $\-\Delta z+2\lambda \partial_xz+(m^2-\lambda^2)z\geq0$.

Another application of the strong maximum principle yields
\[z(0,y_0)=\inf_{\Gamma^+_R}z=\inf_{\Gamma^+_R}\rho=\rho(0,y_0)<0.\]

An application of Hopf's lemma produces $\frac{\partial z}{\partial \eta}(0,y_0)<0$, that is, \begin{equation}\label{z}-\frac{\partial z}{\partial x}(0,y_0)<0.\end{equation} Since $\frac{\partial z}{\partial x}=\frac{\partial \rho}{\partial x}e^{\lambda x}+\lambda \rho e^{\lambda x}$, we conclude that
\[\frac{\partial z}{\partial x}(0,y_0)=\frac{\partial \rho}{\partial x}(0,y_0)+\lambda \rho(0,y_0)\]
and so
\begin{align*}-\frac{\partial z}{\partial x}(0,y_0)
&=-\frac{\partial f_R}{\partial x}(0,y_0)+\frac{\partial w}{\partial x}(0,y_0)-\lambda f_R(0,y_0)+\lambda w(0,y_0)\\
&=(\alpha-\lambda)f_R(0,y_0)+V(y_0)w(0,y_0)-K(y_0)f(w(0,y_0))\\&\qquad+\lambda w(0,y_0)\\
&=(\alpha-\lambda)f_R(0,y_0)+(V(y_0)+V_0)w(0,y_0)-K(y_0)f(w(0,y_0))\\
&\qquad+(\lambda-V_0)w(0,y_0).
\end{align*}
Now, choosing $\alpha=\lambda$, since $\lambda>V_0$ (so that the last term in the above inequality is non-negative), the positiveness of $(V(y_0)+V_0)w(0,y_0)$ and hypothesis $(f_1)$ guarantees that $-\frac{\partial z}{\partial x}(0,y_0)>0$, thus reaching a contradiction with \eqref{z}.
$\hfill\Box$

\end{document}